\documentclass[12pt]{amsart}
\usepackage{amsmath,amsfonts,amssymb}

\theoremstyle{plain}

\newtheorem{cor}{Corollary}
\newtheorem{thm}{Theorem}
\newtheorem{lem}{Lemma}

\title[Integers as combinations of powers]
{Representing integers as linear combinations of powers}

\author{Lajos Hajdu}
\address{Institute of Mathematics,\\
University of Debrecen,\\
and the Number Theory Research Group\\
of the Hungarian Academy of Sciences,\\
H-4010 Debrecen, P.O. Box 12,\\
Hungary}
\email{hajdul@science.unideb.hu}
\author{Rob Tijdeman}
\address{Mathematical Institute,\\
Leiden University,\\
2300 RA Leiden, P.O. Box 9512\\
The Netherlands}
\email{tijdeman@math.leidenuniv.nl}



\thanks{Research supported in part by the OTKA grants K67580
and K75566, and by the T\'AMOP 4.2.1./B-09/1/KONV-2010-0007
project. The project is implemented through the New Hungary
Development Plan, cofinanced by the European Social Fund and the European Regional Development Fund.}

\begin{document}

\dedicatory{Dedicated to Professors K. Gy\H{o}ry and A.
S\'ark\"ozy on their 70th birthdays\\ and Professors A. Peth\H{o} and J. Pintz on their 60th birthdays}

\begin{abstract} At a conference in Debrecen in October
2010 Nathanson announced some results concerning the arithmetic diameters of certain sets. (See his paper in the present
volume.) He proposed some related problems on the representation of integers by sums or differences of powers of $2$ and of $3$. In this note we prove some results on this problem and the more general problem about the representation by  linear combinations of powers of some fixed integers.
\end{abstract}

\maketitle

\section{Introduction}

Let $P$ be a nonempty finite set of prime numbers, and let $T$ be the set of positive integers that are products of powers of primes in $P$. Put $T_P=T\cup (-T)$. Then there does not exist an integer $k$ such that every positive integer can be represented as a sum of at most $k$ elements of $T_P$. This follows e.g. from Theorem 1 of Jarden and Narkiewicz \cite{jn}, cf. \cite{h, ahl}. At a conference in Debrecen in October 2010 Nathanson announced the following stronger result (see also \cite{n}):
\\
\\
{\it For every positive integer $k$ there exist infinitely many integers $n$ such that $k$ is the smallest value of $l$ for which
$n$ can be written as
$$
n=a_1+a_2+\cdots+a_l~~(a_1,a_2,\dots,a_l\in T_P).
$$
}
\\
Let $f(k)$ be the smallest positive integer which cannot be
represented as sum of less than $k$ terms from $T_P$. In Problem 2 of \cite{n} Nathanson asked to give estimates for $f(k)$. (The notation in \cite{n} is somewhat different from ours.) Problem 1 asks the same question in case $T$ consists of the pure powers of $2$ and of $3$. Observe that in both cases $f(k)$ can be represented as a sum of $k$ terms from $T_P$, since less than $T_P$ terms suffice for $f(k)-1$ and $1\in T_P$.

In this note we consider Problem 1. More generally, let $B=\{b_1,\dots,b_t\}$ be any finite set of positive integers. Put $A = \{ b_i^j,-b_i^j~ :~ i=1,\dots,t;j=0,1,2,\dots\}$. Note that on writing $P=\{p\ \text{prime}\ :\ p\mid b_1\cdots b_t\}$ we have $A\subseteq T_P$. So there is no $k$ for which every positive integer can be represented as a sum of at most $k$ elements of $A$. Let $f(k)$ be the smallest positive integer which cannot be represented as sum of less than $k$ terms of $A$. Similarly as above, we get that $f(k)$ can be represented as a sum of $k$ terms of $A$.

In this paper we show that there exists a number $c$ depending
only on $B$ and an absolute constant $C$ such that $\exp(ck)<f(k)<\exp((k \log t)^C)$. Moreover, we show that there are infinitely many $k$'s for which $f(k)<\exp(c^*k\log
(2kt)\log\log k)$ where $c^*$ is some constant.

For the upper bound we apply a method of \'Ad\'am, Hajdu and Luca \cite{ahl} in which a result of Erd\H{o}s, Pomerance and Schmutz \cite{eps} plays an important part. We refine the result of Erd\H{o}s, Pomerance and Schmutz in Section \ref{sec2} and that of \'Ad\'am, Hajdu and Luca in Section \ref{sec3}. In Section \ref{sec4} we derive lower and upper bounds for $f(k)$ in a somewhat more general setting. We conclude with some remarks in Section \ref{sec5}.

\section{An extension of a theorem of Erd\H{o}s, Pomerance and Schmutz}\label{sec2}

Let $\lambda(m)$ be the Carmichael function of the positive
integer $m$, that is the least positive integer for which
$$
b^{\lambda(m)} \equiv 1 ~~({\rm mod}~m)
$$
for all $b \in \mathbb{Z}$ with gcd$(b,m)=1$. Theorem 1 of
\cite{eps} gives the following information on small values of the Carmichael function.

\vskip.1cm

\noindent {\it For any increasing sequence $(n_i)_{i=1}^{\infty}$ of positive integers, and any positive constant $c_0 < 1/ \log 2$, one has
$$
\lambda (n_i) > ( \log n_i)^{c_0 \log \log \log n_i}
$$
for $i$ sufficiently large. On the other hand, there exist a strictly increasing sequence $(n_i)_{i=1}^{\infty}$ of positive integers and a positive constant $c_1$, such that, for every $i$,
$$
\lambda (n_i)<(\log n_i)^{c_1\log\log\log n_i}.
$$}

This nice theorem does not give any information on the size of
$n_i$. Since we need such information in this paper, we prove the following refinement of the second part. The proof is an extension of the proof in \cite{eps}.

\begin{thm}
\label{thm1} There exist positive constants $c_2, c_3$ such that for every large integer $i$ there is an integer $m$ with $\log m\in [\log i,(\log i)^{c_2}]$ and
$$
\lambda(m)<(\log m)^{c_3\log\log\log m}.
$$
\end{thm}

\begin{proof}
In \cite{apr} it is shown that there is a computable constant $c_4>0$ with the property that, for any $x>10$, there is a squarefree number $h_x < x^2$ for which
$$
\sum_{p-1 | h_x} ~ 1>e^{c_4\log x/\log\log x}.
$$
Put $x =(\log i)^{(2/c_4)\log\log\log i},y=h_{x}$, and $m=\prod_{p-1 | y} ~ p$. Note that, for $i$ sufficiently large, we have
$$
m\geq\prod_{p-1| y}2 ~>~ \exp\left((\log~2)\exp\left(\frac{c_4 \log x}{\log\log x}\right)\right)>i.
$$
But then, for $i$ sufficiently large and $c_3=4/c_4$,
$$
\lambda(m)\leq y<x^2=(\log i)^{(4/c_4)\log\log\log i}<
(\log m)^{c_3\log\log\log m}.
$$

It remains to estimate $m$ from above. Let $s$ be the number of prime factors of the squarefree number $y$ and $0<\varepsilon<0.1$. Then $y$ is at least $s^{(1-\varepsilon)s}$ if $i$ is sufficiently large. Hence $s<(1+2\varepsilon)\log y/\log\log y$. It follows that
$$
\sum_{p-1 | y} ~1\leq 2^s<\frac{y^{1/\log\log y}}{\log(y+1)}
$$
when $i$ is large. Thus
$$
\log m=\log\left(\prod_{p-1 | y} ~ p \right)<\sum_{p-1 | y}
~ \log (y+1)<y^{1/\log\log y}< x^{2/\log\log x}<(\log i)^{c_2}
$$
for some constant $c_2$.
\end{proof}

\section{An extension of a theorem of \'Ad\'am, Hajdu and Luca} \label{sec3}

Let $B=\{b_1,\dots,b_t\}$ be any finite set of positive integers. Let $A=\{b_i^j ~:~ i=1,\dots,t;j=0,1,2,\dots\}$. Let $k$ be a positive integer and $R$ a finite set of integers of
cardinality $\rho$. Put
$$
H_{B,R,k}=\{n\in\mathbb{Z}:n=\sum_{i=1}^k r_ia_i\}
$$
where $r_i\in R,a_i\in A ~~ (i=1,2,\dots,k)$. For $H\subseteq\mathbb{Z}$ and $m\in\mathbb{Z},m\geq 2$, we write
$\sharp H$ for the cardinality of the set $H$ and
$$
H({\rm mod}~m)=\{i:0\leq i<m,h\equiv i ~( {\rm mod}~ m)~
{\rm for ~ some~}h\in H\}.
$$
Observe that the definition of $A$ differs from that in the
introduction and that we get the situation described there by
choosing $R=\{-1,1\}$.

\begin{thm} \label{thm2} Let $B,R$ and $k$ be given as above. For every sufficiently large integer $i$ there exists a number $m$ with $\log m\in [\log i,(\log i)^{c_2}]$ such that
$$
\sharp H_{B, R, k} ~({\rm mod}~m)<(\rho t)^k(\log m)^{c_5k\log\log\log m}
$$
where $c_5$ is a constant.
\end{thm}

In the proof of Theorem \ref{thm1} the following lemma is used.

\begin{lem} \label{lem1} {\rm (\cite{ahl}, Lemma 1)}.\\
Let $m=q_1^{\alpha_1}\cdots q_z^{\alpha_z}$ where $q_1,\dots,q_z$ are distinct primes and $\alpha_1,\dots,\alpha_z$ are positive integers, and let $b\in\mathbb{Z}$. Then
$$
\sharp\{b^u~({\rm mod}~m):u\geq 0\}\leq\lambda(m)+
\max_{1\leq j\leq z}\alpha_j.
$$
\end{lem}

The proof of Theorem \ref{thm2} is similar to that of Theorem 3 of \cite{ahl}. In that paper there is the restriction that of each element of $B$ only one power occurs in $H_{B,R,K}$, hence $k=t$.

\begin{proof}[Proof of Theorem \ref{thm2}] Let $i$ be an integer so large that Theorem \ref{thm1} applies. Choose $m$ as in Theorem \ref{thm1}. Write $m$ as in Lemma \ref{lem1} as a product of powers of distinct primes. Lemma \ref{lem1} implies that for all $b\in B$,
$$
\sharp\{r\cdot b^u~({\rm mod}~m):b\in B,r\in R,u\geq 0 \}
\leq\rho t\left(\lambda(m)+\max_{1\leq j\leq z}\alpha_j\right).
$$
On the other hand, with the constant $c_3$ from Theorem \ref{thm1},
$$
\lambda(m)+\max_{1\leq j\leq z}\alpha_j\leq(\log m)^{c_3
\log\log\log m}+\frac{\log m}{\log 2}.
$$
The combination of both inequalities yields the theorem.
\end{proof}

\section{Representing integers as linear combinations of powers} \label{sec4}

We use the notation of Section \ref{sec3}. Suppose we want to
express the positive integer $n$ as a finite sum of powers of
$b_1$. For this we  apply the greedy algorithm. If we subtract the largest power of $b_1$ not exceeding $n$ from $n$, we obtain a number which is less than $n(1-1/b_1)$. We can iterate subtracting the highest power of $b_1$ not exceeding the rest from the rest and so reduce the rest each time by a factor at most $1-1/b_1$. Hence we can represent $n$ as the sum of at most $\log n/\log(1/(1-1/b_1))$ powers of $b_1$. Thus we find that the sum of $k\leq c_6\log n$ powers of $b_1$ suffices to represent $n$, where $c_6$ depends only on $b_1$. This implies the lower bound $\exp(ck)$ for $f(k)$ claimed in the introduction. More generally, let $f_R(k)$ be the smallest positive integer $n$ which cannot be represented as a sum $\sum_{j=1}^lr_ja_j$ with $l<k,r_j\in R,a_j\in A$. Then the above argument shows that $1\in R$ implies $f_R(k)>e^{k/c_6}$.

For an upper bound for $f_R(k)$ suppose first that all the elements of $R$ are positive. We study the representation of positive integers up to $n$ as $\sum_{j=1}^{k-1}r_jb_j^{k_j}$ with $r_j\in R,b_j\in B,k_j\in\mathbb{Z},k_j\geq 0$. Then $k_j\leq\log n/\log b_j\leq\log n/\log 2$. Hence the number of
represented integers is at most $\left(\rho t\log n/\log
2\right)^{k-1}$. If this number is less than $n$, then we are sure that some positive integer $\leq n$ is not represented. This is the case if
$$
k-1<\frac{\log n}{\log(\rho t)+\log\log n-\log\log 2}.
$$
Hence it suffices that $n\geq(1.5\rho kt\log(\rho
kt))^{k-1}$ and for this special case we find that
$$
f_R(k)\leq(1.5\rho kt\log (\rho kt))^{k-1}.
$$

We now turn to the general case. Choose the smallest positive
integer $i>10$ such that $i>(\rho t)^k(\log i)^{c_5k\log\log\log i}$. Then $i<2(\rho t)^k(\log i)^{c_5k\log\log\log i}$. It follows that
$$
\log i<k(\log\rho t)+c_7k(\log\log i)(\log\log\log i)
$$
for some constant $c_7$, thus $\log i<2k\log(\rho t)$ or $\log i<2c_7k(\log\log i)(\log\log\log i)$. In the latter case $\log i<c_{8}k(\log k)(\log\log k)$ for some suitable constant $c_{8}$. According to Theorem \ref{thm2} there exists an $m$ with $\log i\leq\log m\leq(\log i )^{c_2}$ such that all representations  are covered by at most  $(\rho t)^k (\log
m)^{c_5k\log\log\log m}$ residue classes modulo $m$. By the
definition of $i$ and the inequality $i\leq m$, we see that this number of residue classes is less than $m$, therefore at least one positive integer $n\leq m$ has no representation of the form $\sum_{j=1}^{k}r_ja_j$ with $r_j\in R,a_j\in A$ for all $j$. Since $\log m\leq(\log i)^{c_2}$, we obtain
$$
\log n\leq\log m \leq(\log i)^{c_2}<\left(\max(2k\log
\rho t,c_{8}k(\log k)(\log\log k))\right)^{c_2}<(k\log\rho t)^{c_{9}}
$$
for some constant $c_{9}$.

There are infinitely many $k$'s for which a considerably better bound for $f_R(k)$ can be derived by a variant of the above
argument. According to Theorem \ref{thm1} there are infinitely
many integers $m$ for which
\begin{equation}
\label{eps}
\lambda(m)<(\log m)^{c_3\log\log\log m}.
\end{equation}
Let $B$, hence $A,\rho$ and $t$ be given. Choose $k$ as the
largest integer such that
$$
(\rho t)^k(\log m)^{c_5k\log\log\log m}<m
$$
for an $m$ satisfying (\ref{eps}). It follows from Theorem \ref{thm1} that there are infinitely many such $k$'s. Theorem \ref{thm2} and its proof imply that there is a positive integer $n\leq m$ which is not representable as a linear combination of $k$ elements of $A$ with coefficients from $R$. Moreover,
$$
\log m\leq(k+1)(\log(\rho t)+c_5\log\log m\log\log\log m).
$$
Hence $\log m\leq2(k+1)\log(\rho t)$ or $\log m\leq 2c_5(k+1)\log\log m\log\log\log m$. In the latter case $\log m\leq c_{10}k\log k\log\log k$ where $c_{10}$ is some
constant. Combining both inequalities we obtain, for some constant $c_{11}$,
$$
\log n\leq\log m\leq c_{11}k\log(\rho kt)\log\log k.
$$

So we have proved the following result.

\begin{thm} \label{thm3} Let $B=\{b_1,\dots,b_t\}$ be any finite set of positive integers. Put $A=\{b_i^j ~:~ i=1,\dots,t;j=0,1,2,\dots\}$. Let $R$ be a finite set of integers of cardinality $\rho$ and $k$ a positive integer. Denote by $f_R(k)$ the smallest positive integer which cannot be represented in the form $ \sum_{i=1}^{k-1}r_ia_i$ with $r_i\in R,a_i\in A$ for all $i$. Then\\
(i) if $1\in R$, then $\log f_R(k)>k/c_6$ for some number $c_6>0$ depending only on $b_1$,\\
(ii)  if all elements of $R$ are positive, then $f_R(k)\leq(1.5\rho kt\log(\rho kt))^{k-1}$,\\
(iii) there exists a constant $c_9$ such that $\log f_R(k)<(k\log(\rho t))^{c_9}$,\\
(iv) there exist a constant $c_{11}$ and infinitely many positive integers $k$ such that $\log f_R(k)\leq c_{11}k\log(\rho kt)\log\log k$.
\end{thm}

In Nathanson's Problem 1 mentioned in the introduction we have $R=\{-1,1\}$, and hence $\rho=2$. Thus we have the following
consequences for the function $f$.

\begin{cor} There is a positive number $c$ depending only on $b_1$ such that $\log f(k)>ck$.\\
On the other hand, $\log f(k)<(k\log (2t))^C$ where $C$ is a
constant.\\
Moreover, $\log f(k)<c^*k \log (2kt)\log\log k$ for infinitely many integers $k$ where $c^*$ is a constant.
\end{cor}

\section{Some remarks} \label{sec5}

\noindent{\bf Remark 1.} To prove that $f_R(k)>e^{ck}$ we assumed $1\in R$. Here we check what happens if this condition is not fulfilled. Obviously, we may assume that $0\notin R$ and that not all the elements of $R$ are negative. Further, if the elements of $R$ are not coprime, then there is a full residue class not represented as $\sum_{i=1}^k r_ia_i$. Therefore we may assume that the elements of $R$ are coprime. (In particular, since $R=\{1\}$ is now excluded, that $\rho>1$.)

Assume first that $R$ contains a negative element. There exist
integers $d_j~ (j=1,\dots,\rho)$ such that $\sum_{j=1}^\rho
d_jr_j=1$. Let $P=\prod_{j=1}^\rho |r_j|$. So $P$ is a multiple of $r_j$ for every $j$. Consider the sum $\sum_{j=1}^{\rho} \left(d_j+e_j\frac{P}{r_j}\right)r_j$ where the $e_j$'s are integers such that $d_j+e_j\frac{P}{r_j}\geq 0$ and $e_jr_j\geq 0$ for all $j$. Let $v=\sum_{j=1}^t e_j$. If $v>0$ then we replace $e_j$ by $e_j-v$ for some $j$ with $r_j<0$, and if $v<0$ then we do so for some $j$ with $r_j>0$. Afterwards $\sum_{j=1}^t e_j=0$, hence $\sum_{j=1}^{\rho} \left(d_j+e_j\frac{P}{r_j}\right)r_j=1$, and further
$d_j+e_j\frac{P}{r_j} \geq 0$ for $j=1,\dots,\rho$. Thus $1$
admits a representation $\sum_{j=1}^k r_ja_j$ where $a_j=b_1^0=1$ for all $j$, and $k$ is bounded by $\sum_{j=1}^{\rho}\left(d_j+e_j\frac{P}{r_j}\right)$, that is by a constant $c_{12}$ which only depends on $R$.

If the elements of $R$ are coprime and all positive, then we have to do with the so-called coin problem or Frobenius problem. Let $0<r_1\leq r_2\leq\dots\leq r_{\rho}$. Schur \cite{b} proved in 1935 that every number larger than $c_{13}:=r_1r_{\rho}+r_2+\dots+r_{\rho-1}$ can be represented as a linear combination of $r_1,r_2,\dots,r_{\rho}$ with nonnegative integer coefficients. Note that $c_{13}$ depends only on $R$. Therefore each of the integers in the interval $(c_{13},2c_{13}]$ can be represented as $\sum_{j=1}^{c_{14}} r_ja_j$ where the number $c_{14}$ depends only on $R$. We can now use the greedy algorithm as in the first block of Section \ref{sec4}, iterating until we reach this interval, to obtain a representation of $n>c_{13}$ with at most $(c_{6}+c_{14})\log n$ terms $r_ja_j$.

We conclude that if the elements of $R$ are coprime, every
positive integer $n>c_{13}$ can be represented as $\sum_{i=1}^{k-1} r_ia_i$ with $r_i\in R,a_i\in A$ for all $i$
and with $k<c_{15}\log n$, where $c_{15}$ is a number depending only on $R$. Thus $f_R(k)>e^{ck}$ where $c$ is a number depending only on $R$.

\vskip.1truecm

\noindent {\bf Remark 2.} The upper bound for $f(k)$ can possibly be improved by deriving a version of Theorem \ref{thm1} where the interval for $m$ is essentially smaller at the cost of a larger bound for $\lambda(m)$. We expect that the given upper bound for infinitely many values of $k$ may be close to an upper bound for all $k$.

\section{Acknowledgements}

The authors are grateful to the referees for their helpful remarks.

\end{document}